\documentclass[11pt]{amsart}
\usepackage[T1]{fontenc}
\usepackage[utf8]{inputenc}
\usepackage{lmodern}
\usepackage{microtype}
\usepackage[margin=1.15in]{geometry}
\usepackage{amsmath,amssymb,amsthm,mathtools}
\usepackage{graphicx}
\usepackage{enumitem}
\usepackage{xurl}
\usepackage{hyperref}
\usepackage{aliascnt}
\usepackage[nameinlink,noabbrev]{cleveref}
\graphicspath{{figures/}}
\hypersetup{
  hidelinks,
  pdftitle={Explicit interpolations among the Sierpinski, Rauzy, and Apollonian gaskets},
  pdfauthor={Bernat Espigule},
  pdfsubject={Research article on explicit interpolations among triangular fractals},
  pdfkeywords={Sierpinski gasket, Rauzy gasket, Apollonian gasket, projective iterated function system, Mobius transformation, Hausdorff dimension, finite approximants}
}
\newtheorem{theorem}{Theorem}[section]
\newaliascnt{proposition}{theorem}
\newtheorem{proposition}[proposition]{Proposition}
\aliascntresetthe{proposition}
\newaliascnt{lemma}{theorem}
\newtheorem{lemma}[lemma]{Lemma}
\aliascntresetthe{lemma}
\newaliascnt{corollary}{theorem}
\newtheorem{corollary}[corollary]{Corollary}
\aliascntresetthe{corollary}
\theoremstyle{definition}
\newaliascnt{definition}{theorem}
\newtheorem{definition}[definition]{Definition}
\aliascntresetthe{definition}
\newaliascnt{remark}{theorem}
\newtheorem{remark}[remark]{Remark}
\aliascntresetthe{remark}
\newaliascnt{example}{theorem}

\aliascntresetthe{example}

\crefname{theorem}{theorem}{theorems}
\Crefname{theorem}{Theorem}{Theorems}
\crefname{proposition}{proposition}{propositions}
\Crefname{proposition}{Proposition}{Propositions}
\crefname{lemma}{lemma}{lemmas}
\Crefname{lemma}{Lemma}{Lemmas}
\crefname{corollary}{corollary}{corollaries}
\Crefname{corollary}{Corollary}{Corollaries}
\crefname{remark}{remark}{remarks}
\Crefname{remark}{Remark}{Remarks}
\crefname{definition}{definition}{definitions}
\Crefname{definition}{Definition}{Definitions}

\newcommand{\R}{\mathbb{R}}
\newcommand{\C}{\mathbb{C}}
\newcommand{\N}{\mathbb{N}}
\newcommand{\D}{\Delta}
\newcommand{\Disk}{\overline{\mathbb{D}}}

\newcommand{\HH}{\mathcal{H}}
\newcommand{\K}{\mathcal{K}}

\newcommand{\diam}{\operatorname{diam}}

\newcommand{\dimH}{\operatorname{dim_{H}}}
\newcommand{\eps}{\varepsilon}
\newcommand{\ii}{\mathrm{i}}

\title[Explicit gasket interpolations]{Explicit interpolations among the Sierpi\'nski, Rauzy, and Apollonian gaskets}
\author{Bernat Espigule}
\address{Universitat de Girona\\Girona, Catalonia, Spain}
\email{bernat@espigule.com}
\date{June 2026}
\subjclass[2020]{Primary 28A80; Secondary 37F35, 37C45, 52C26}
\keywords{Sierpi\'nski gasket, Rauzy gasket, Apollonian gasket, projective iterated function system, M\"obius transformation, Hausdorff dimension, finite approximants, fractal geometry}

\begin{document}
\begin{abstract}
We study two explicit one-parameter families organized around the affine Sierpi\'nski gasket. The first is an affine-projective interpolation from the Sierpi\'nski gasket to the Rauzy gasket: the first-level hole is fixed throughout the family, the symbolic quotient remains the classical Sierpi\'nski quotient, the associated three-dimensional stack is homeomorphic to $S\times[0,1]$, and the constant-address cells display a transition from uniform contraction to the non-uniform projective scaling of the Rauzy endpoint. The second is a M\"obius deformation from the Sierpi\'nski gasket to the equilateral Apollonian gasket: for every hyperbolic parameter the maps form a conformal iterated function system on a common disk, vary continuously on compact subintervals away from the parabolic endpoint, and admit an exact variational formula for the Hausdorff dimension. At the parabolic endpoint we give an explicit boundary-normalized Apollonian model whose distinguished side is the segment $[0,1]$; in this normalization two endpoint branches are $z/(z+1)$ and $1/(2-z)$, the same fractional transformations that occur on the distinguished Rauzy side. This yields a canonical Rauzy--Apollonian homeomorphism fixing the common side pointwise. The constructions provide an exact framework for comparing projective, affine, and conformal triangular fractals through explicit maps, finite-level approximants, and three-dimensional embeddings.
\end{abstract}
\maketitle

\section{Introduction}
The Sierpi\'nski gasket is a natural meeting point between several triangular fractal geometries \cite{Sierpinski1915,Hutchinson1981}. In one direction it appears as the affine representative of the symbolic quotient underlying the Rauzy gasket, a projective fractal from multidimensional continued fractions and Arnoux--Rauzy dynamics \cite{ArnouxStarosta2013,Jurga2026}. In another direction it appears as the linear endpoint of a deformation toward the equilateral Apollonian gasket, where the straight triangular gaps bend into mutually tangent circular gaps \cite{McMullen1998,VytnovaWormell2025}.

This paper treats two explicit deformations of the affine Sierpi\'nski gasket. In the affine-projective direction, the endpoint at $t=1$ is the Rauzy gasket and the endpoint at $t=0$ is the usual Sierpi\'nski system. In the conformal direction, an explicit one-parameter family of M\"obius maps reaches the equilateral Apollonian gasket at a parabolic endpoint.

The two deformations are not dynamically identical. The Rauzy endpoint is projective and non-uniformly scaled, whereas the Apollonian endpoint is conformal and parabolic. The common role of the Sierpi\'nski gasket is structural: it supplies the affine model from which both endpoint geometries are reached by exact parameterized maps.

At the Apollonian endpoint we also record a boundary normalization: the same gasket is M\"obius equivalent to a model in the upper half-plane whose distinguished side is $[0,1]$ and whose adjacent outer circles have radius $1/2$ and centers $\ii/2$ and $1+\ii/2$. On this side the branches are $z\mapsto z/(z+1)$ and $z\mapsto1/(2-z)$, exactly as on the corresponding Rauzy side.

The purpose of the paper is to keep the comparison at the level of explicit maps, symbolic codings, finite-level approximants, and the associated three-dimensional embeddings. The figures should therefore be read as finite-level approximants of the systems defined in the corresponding sections.

\section{The Rauzy--Sierpi\'nski affine-projective stack}
This section gives the affine-projective construction. We first define the maps on the standard simplex, then identify the endpoint systems and record the symbolic and topological consequences needed for the stack construction.

\subsection{The interpolating family}

\subsubsection{Simplex coordinates}

Let
\[
\D:=\bigl\{(x_1,x_2,x_3)\in\R_{\ge 0}^3:x_1+x_2+x_3=1\bigr\}
\]
be the standard $2$-simplex, with vertices
\[
 e_1=(1,0,0),\qquad e_2=(0,1,0),\qquad e_3=(0,0,1).
\]
We use the affine chart
\[
\iota:T\to\D,\qquad \iota(u,v)=(1-u-v,u,v),
\]
where
\[
T:=\bigl\{(u,v)\in\R^2:u\ge0,\ v\ge0,\ u+v\le1\bigr\}.
\]
In this chart the vertices are
\[
 e_1=(0,0),\qquad e_2=(1,0),\qquad e_3=(0,1).
\]
For display purposes we also use the affine equilateral-triangle embedding
\[
\mathcal E(u,v):=\left(u+\frac{v}{2}-\frac12,\frac{\sqrt 3}{2}v\right)\in\R^2,
\]
which sends $e_1,e_2,e_3$ to the vertices of an equilateral triangle. This changes only the presentation of the figures: all formulas and proofs continue to use the chart $T$.
For $i\neq j$ let
\[
 m_{ij}:=\frac{e_i+e_j}{2}.
\]
We also write
\[
T_1:=\{u+v\le 1/2\},\qquad T_2:=\{u\ge 1/2\},\qquad T_3:=\{v\ge 1/2\},
\]
so that $T=T_1\cup T_2\cup T_3\cup H$, where the open central hole is
\[
H:=\{u<1/2,\ v<1/2,\ u+v>1/2\}.
\]
See \cref{fig:firstlevel}.

\begin{figure}[t]
    \centering
    \includegraphics[width=0.46\textwidth]{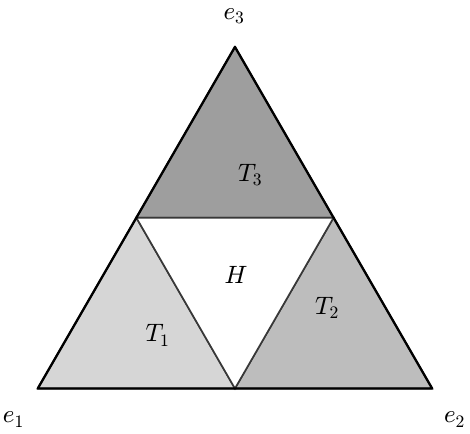}
    \caption{The parameter-independent first-level decomposition $T=T_1\cup T_2\cup T_3\cup H$, displayed through the equilateral embedding $\mathcal E$. The central hole $H$ is the same for every $t\in[0,1]$.}
    \label{fig:firstlevel}
\end{figure}

\subsubsection{Definition of the maps}

For $t\in[0,1]$ define matrices
\[
M_1(t)=\begin{pmatrix}2-t&1&1\\0&1&0\\0&0&1\end{pmatrix},\quad
M_2(t)=\begin{pmatrix}1&0&0\\1&2-t&1\\0&0&1\end{pmatrix},\quad
M_3(t)=\begin{pmatrix}1&0&0\\0&1&0\\1&1&2-t\end{pmatrix}.
\]
Since $\det M_i(t)=2-t>0$, each $M_i(t)$ is invertible. Let $\mathbf 1=(1,1,1)^\top$ and set
\[
 f_{i,t}(x):=\frac{M_i(t)x}{\mathbf 1^\top M_i(t)x},\qquad x\in\D.
\]
In the chart $T$ these maps become
\begin{align}
 f_{1,t}(u,v)&=\left(\frac{u}{2-t+t(u+v)},\frac{v}{2-t+t(u+v)}\right),\label{eq:f1chart}\\
 f_{2,t}(u,v)&=\left(\frac{1+(1-t)u}{2-tu},\frac{v}{2-tu}\right),\label{eq:f2chart}\\
 f_{3,t}(u,v)&=\left(\frac{u}{2-tv},\frac{1+(1-t)v}{2-tv}\right).\label{eq:f3chart}
\end{align}
At $t=0$ one recovers the affine Sierpi\'nski maps
\[
 g_1(u,v)=\left(\frac u2,\frac v2\right),\quad
 g_2(u,v)=\left(\frac{1+u}{2},\frac v2\right),\quad
 g_3(u,v)=\left(\frac u2,\frac{1+v}{2}\right),
\]
while at $t=1$ one obtains the classical Rauzy inverse branches from \cite[\S3.1]{ArnouxStarosta2013}.

\begin{proposition}[Parameter-independent first-level geometry]\label{prop:firstlevel}
For every $t\in[0,1]$ one has
\[
 f_{1,t}(T)=T_1,\qquad f_{2,t}(T)=T_2,\qquad f_{3,t}(T)=T_3.
\]
Equivalently, in barycentric coordinates,
\[
 f_{i,t}(\D)=\D_i:=\{x\in\D:x_i\ge 1/2\}.
\]
Moreover,
\[
 T_i\cap T_j=\{m_{ij}\}\qquad (i\neq j).
\]
\end{proposition}

\begin{proof}
For $f_{1,t}$, let $s=u+v$. Then \eqref{eq:f1chart} gives
\[
 f_{1,t}(u,v)=\frac{1}{2-t+ts}(u,v),\qquad s' := u'+v' = \frac{s}{2-t+ts}.
\]
As $s$ varies in $[0,1]$, the function $s\mapsto s/(2-t+ts)$ is increasing and maps $[0,1]$ onto $[0,1/2]$. The ratio $v/u$ is preserved, so every point of $T_1$ is reached along the unique ray from the origin through that point. Hence $f_{1,t}(T)=T_1$.

The proofs for $f_{2,t}$ and $f_{3,t}$ are identical after permuting coordinates. The intersection statement is then immediate from the explicit descriptions of $T_1,T_2,T_3$.
\end{proof}

The piecewise inverse map $F_t$ on $T_1\cup T_2\cup T_3$ is therefore well defined by
\[
F_t|_{T_i}=f_{i,t}^{-1}.
\]
Explicitly,
\begin{align*}
F_t(u,v)=
\begin{cases}
\left(\dfrac{(2-t)u}{1-t(u+v)},\dfrac{(2-t)v}{1-t(u+v)}\right),& (u,v)\in T_1,\\[1.1ex]
\left(\dfrac{2u-1}{1-t+tu},\dfrac{(2-t)v}{1-t+tu}\right),& (u,v)\in T_2,\\[1.1ex]
\left(\dfrac{(2-t)u}{1-t+tv},\dfrac{2v-1}{1-t+tv}\right),& (u,v)\in T_3.
\end{cases}
\end{align*}
At $t=0$ this is the usual expanding map for the Sierpi\'nski gasket; at $t=1$ it is the Rauzy map.

\subsection{Attractors in the contracting regime}

\begin{proposition}[Uniform contraction for $t<1$]\label{prop:contraction}
Fix $t\in[0,1)$ and set
\[
q_t:=\frac{1}{2-t}<1.
\]
Then each map $f_{i,t}:T\to T$ is $q_t$-Lipschitz for the norm
\[
\|(u,v)\|_\infty:=\max\{|u|,|v|\}.
\]
Consequently, the Hutchinson operator
\[
\HH_t(X):=f_{1,t}(X)\cup f_{2,t}(X)\cup f_{3,t}(X)
\]
is a strict contraction on the space of nonempty compact subsets of $T$ endowed with the Hausdorff metric induced by $\|\cdot\|_\infty$.
\end{proposition}

\begin{proof}
For $f_{1,t}$ let $D=2-t+t(u+v)$. Differentiating \eqref{eq:f1chart} yields
\[
Df_{1,t}(u,v)=\frac{1}{D^2}
\begin{pmatrix}
2-t+tv&-tu\\
-tv&2-t+tu
\end{pmatrix}.
\]
The absolute row sums are
\[
\frac{2-t+t(u+v)}{D^2}=\frac1D\le \frac{1}{2-t}=q_t.
\]
Thus $\|Df_{1,t}(u,v)\|_{\infty\to\infty}\le q_t$.

For $f_{2,t}$ let $E=2-tu$. Differentiating \eqref{eq:f2chart} gives
\[
Df_{2,t}(u,v)=\frac{1}{E^2}
\begin{pmatrix}
2-t&0\\
tv&E
\end{pmatrix}.
\]
The first absolute row sum is $(2-t)/E^2\le 1/(2-t)$. For the second row,
\[
\frac{tv+E}{E^2}\le \frac{E+(E-(2-t))}{E^2}=\frac{2E-(2-t)}{E^2}\le \frac{1}{2-t},
\]
because $tv\le t(1-u)=E-(2-t)$ and the function $x\mapsto (2x-(2-t))/x^2$ is decreasing on $[2-t,2]$. The proof for $f_{3,t}$ is symmetric.

The claim about $\HH_t$ follows from the usual Hausdorff-metric estimate for finite unions of Lipschitz maps.
\end{proof}

\begin{definition}
For $t\in[0,1)$, let $K_t\subset T$ denote the unique nonempty compact set satisfying
\[
K_t=f_{1,t}(K_t)\cup f_{2,t}(K_t)\cup f_{3,t}(K_t).
\]
We also set $K_0=:S$, the Sierpi\'nski gasket.
\end{definition}

\begin{figure}[t]
    \centering
    \includegraphics[width=\textwidth]{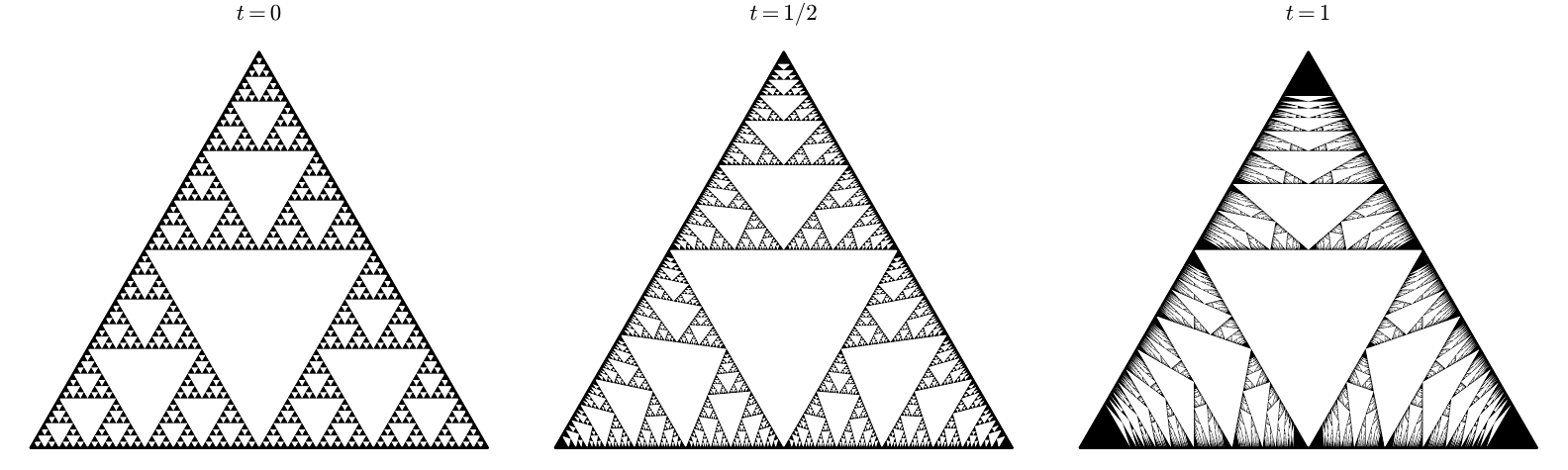}
    \caption{Finite-level approximants for $t=0$, $t=1/2$, and $t=1$, displayed in the same equilateral embedding $\mathcal E$. The first-level hole is unchanged, while the geometry of deeper cells passes from affine to projective. In the Rauzy endpoint panel, the larger black triangular gaps are finite-level cylinder complements, left visible to emphasize the non-homogeneous projective scaling of the Rauzy system.}
    \label{fig:crosssections}
\end{figure}

\subsection{Coding and topological constancy below the Rauzy endpoint}

Let $\Sigma:=\{1,2,3\}^{\N}$ be the full one-sided shift. For a finite word
\[
w=i_1\cdots i_n\in\{1,2,3\}^{<\infty}
\]
we write
\[
f_{w,t}:=f_{i_1,t}\circ\cdots\circ f_{i_n,t}.
\]

\begin{proposition}[Coding maps]\label{prop:coding}
For every $t\in[0,1)$ and every $\omega=(\omega_n)_{n\ge1}\in\Sigma$, the limit
\[
\pi_t(\omega):=\lim_{n\to\infty} f_{\omega_1,t}\circ\cdots\circ f_{\omega_n,t}(x_\ast)
\]
exists, is independent of the base point $x_\ast\in T$, and belongs to $K_t$. The map $\pi_t:\Sigma\to K_t$ is continuous and onto.
\end{proposition}

\begin{proof}
Fix $x_\ast\in T$. By Proposition~\ref{prop:contraction}, for every $n,p\ge0$,
\[
\bigl\|f_{\omega_1,t}\circ\cdots\circ f_{\omega_{n+p},t}(x_\ast)-f_{\omega_1,t}\circ\cdots\circ f_{\omega_n,t}(x_\ast)\bigr\|_\infty
\le q_t^n\,\diam(T),
\]
so the sequence is Cauchy and the limit exists. The same estimate shows independence of $x_\ast$ and continuity of $\pi_t$. Surjectivity is standard for attractors of contractive IFSs.
\end{proof}

\begin{definition}
Define an equivalence relation $\sim$ on $\Sigma$ by declaring that
\[
 w i j^{\infty}\sim w j i^{\infty}\qquad (i\neq j,\; w\in\{1,2,3\}^{<\infty})
\]
and taking the smallest equivalence relation generated by these pairs.
\end{definition}

\begin{theorem}[Parameter-independent address relation]\label{thm:addressrelation}
For every $t\in[0,1)$ and every $\omega,\eta\in\Sigma$ one has
\[
\pi_t(\omega)=\pi_t(\eta)\quad\Longleftrightarrow\quad \omega\sim\eta.
\]
In particular, the quotient $Q:=\Sigma/\!\sim$ is canonically homeomorphic to $K_t$ for every $t<1$.
\end{theorem}

\begin{proof}
The implication $\omega\sim\eta\Rightarrow \pi_t(\omega)=\pi_t(\eta)$ is immediate from
\[
 f_{i,t}(e_j)=m_{ij}=f_{j,t}(e_i),\qquad f_{i,t}(e_i)=e_i.
\]
Indeed,
\[
\pi_t(w i j^\infty)=f_{w,t}(f_{i,t}(e_j))=f_{w,t}(m_{ij})=f_{w,t}(f_{j,t}(e_i))=\pi_t(w j i^\infty).
\]

Conversely, assume $\pi_t(\omega)=\pi_t(\eta)$ and $\omega\neq\eta$. Let $w$ be their maximal common prefix, and write
\[
\omega=w i\alpha,\qquad \eta=w j\beta,
\]
with $i\neq j$. Then
\[
\pi_t(\omega)\in f_{w,t}(T_i)\cap f_{w,t}(T_j).
\]
By Proposition~\ref{prop:firstlevel}, $T_i\cap T_j=\{m_{ij}\}$, and since $f_{w,t}$ is injective we obtain
\[
\pi_t(\omega)=\pi_t(\eta)=f_{w,t}(m_{ij}).
\]
Applying $f_{w,t}^{-1}$ and then $f_{i,t}^{-1}$ yields
\[
\pi_t(\alpha)=e_j.
\]
But $e_j\in T_j$ and $e_j\notin T_k$ for $k\neq j$, so every coding of $e_j$ must begin with $j$. Iterating this argument shows that the only coding of $e_j$ is $j^\infty$. Hence $\alpha=j^\infty$ and similarly $\beta=i^\infty$. Therefore $\omega\sim\eta$.
\end{proof}

\begin{corollary}[Canonical homeomorphisms with the Sierpi\'nski gasket]\label{cor:homeos}
For each $t\in[0,1)$ there exists a canonical homeomorphism
\[
H_t:S\to K_t
\]
obtained by matching symbolic addresses. Equivalently,
\[
H_t\bigl(\pi_0(\omega)\bigr)=\pi_t(\omega)\qquad (\omega\in\Sigma).
\]
\end{corollary}

\begin{proof}
By Theorem~\ref{thm:addressrelation}, both $\pi_0$ and $\pi_t$ factor through the same quotient $Q=\Sigma/\!\sim$. The induced maps $\bar\pi_0:Q\to S$ and $\bar\pi_t:Q\to K_t$ are continuous bijections from a compact space to Hausdorff spaces, hence homeomorphisms. Set $H_t=\bar\pi_t\circ\bar\pi_0^{-1}$.
\end{proof}

\begin{remark}[The Rauzy endpoint]\label{rem:rauzyendpoint}
At $t=1$ the system ceases to be uniformly contractive, but the endpoint fiber is classical: by \cite[Theorem~2 and \S4]{ArnouxStarosta2013}, the unique invariant compact set is the Rauzy gasket $R$, and the corresponding coding relation is again the standard Sierpi\'nski one. We write
\[
K_1:=R,
\]
and denote by $\pi_1:\Sigma\to R$ the Arnoux--Starosta coding map, which factors through the same quotient $Q=\Sigma/\!\sim$.
\end{remark}

\subsection{Global continuity up to the Rauzy endpoint}

For a finite word $w\in\{1,2,3\}^{<\infty}$ and $t\in[0,1]$, let
\[
C_w(t):=f_{w,t}(T).
\]
For an infinite sequence $\omega\in\Sigma$ we write $\omega|_n$ for its prefix of length $n$.

\begin{lemma}[Finite cylinders vary continuously]\label{lem:finitecylinders}
For every finite word $w$, the map $t\mapsto C_w(t)$ is continuous on $[0,1]$ for the Hausdorff metric.
\end{lemma}

\begin{proof}
The map $(x,t)\mapsto f_{w,t}(x)$ is continuous on the compact set $T\times[0,1]$. Hence it is uniformly continuous. Standard compactness then implies that the images $f_{w,t}(T)$ vary continuously in the Hausdorff metric.
\end{proof}

\begin{theorem}[Continuity of the coding map on the full parameter interval]\label{thm:globalcodingcontinuity}
Define
\[
\Pi: \Sigma\times[0,1]\to T,\qquad \Pi(\omega,t):=\pi_t(\omega),
\]
with $\pi_t$ as in Proposition~\ref{prop:coding} for $t<1$ and $\pi_1$ as in Remark~\ref{rem:rauzyendpoint}. Then $\Pi$ is continuous on the whole compact space $\Sigma\times[0,1]$.
\end{theorem}

\begin{proof}
On $\Sigma\times[0,1-\eps]$, continuity follows from the uniform contraction argument used in Proposition~\ref{prop:coding}: the finite approximants
\[
\Pi_n(\omega,t):=f_{\omega_1,t}\circ\cdots\circ f_{\omega_n,t}(x_\ast)
\]
converge uniformly to $\Pi$ on every set of the form $\Sigma\times[0,1-\eps]$.

It remains to prove continuity at points $(\omega,1)$. Fix such a point and let $\varepsilon>0$. Since the Rauzy coding identifies the same quotient $Q=\Sigma/\!\sim$ by Remark~\ref{rem:rauzyendpoint}, the nested compact cylinders
\[
C_{\omega|_1}(1)\supset C_{\omega|_2}(1)\supset \cdots
\]
have singleton intersection $\{\pi_1(\omega)\}$. Therefore
\[
\diam C_{\omega|_n}(1)\longrightarrow 0.
\]
Choose $n$ so large that
\[
\diam C_{\omega|_n}(1)<\varepsilon/3.
\]
By Lemma~\ref{lem:finitecylinders}, there exists $\delta>0$ such that
\[
d_H\bigl(C_{\omega|_n}(t),C_{\omega|_n}(1)\bigr)<\varepsilon/3
\]
whenever $|t-1|<\delta$.

Now let $\eta\in\Sigma$ have the same prefix of length $n$ as $\omega$, and let $|t-1|<\delta$. Then
\[
\pi_t(\eta)\in C_{\eta|_n}(t)=C_{\omega|_n}(t).
\]
Hence there exists $y\in C_{\omega|_n}(1)$ such that
\[
\|\pi_t(\eta)-y\|_\infty<\varepsilon/3.
\]
Since both $y$ and $\pi_1(\omega)$ lie in $C_{\omega|_n}(1)$, we also have
\[
\|y-\pi_1(\omega)\|_\infty\le \diam C_{\omega|_n}(1)<\varepsilon/3.
\]
Therefore
\[
\|\pi_t(\eta)-\pi_1(\omega)\|_\infty<2\varepsilon/3.
\]
This proves continuity at $(\omega,1)$.
\end{proof}

\begin{corollary}[Hausdorff continuity of the fibers on the full interval]\label{cor:globalhausdorff}
The map
\[
[0,1]\to \K(T),\qquad t\mapsto K_t,
\]
is continuous for the Hausdorff metric, where $\K(T)$ denotes the space of nonempty compact subsets of $T$.
\end{corollary}

\begin{proof}
By Theorem~\ref{thm:addressrelation}, Remark~\ref{rem:rauzyendpoint}, and Theorem~\ref{thm:globalcodingcontinuity}, the map $\Pi$ is continuous and constant on the equivalence classes of $\sim$ for every $t\in[0,1]$. It therefore descends to a continuous map
\[
\bar\Pi:Q\times[0,1]\to T,\qquad \bar\Pi(q,t)=\bar\pi_t(q),
\]
whose image at time $t$ is precisely $K_t$. Since $Q\times[0,1]$ is compact, $\bar\Pi$ is uniformly continuous, so
\[
d_H(K_s,K_t)\le \sup_{q\in Q}\|\bar\Pi(q,s)-\bar\Pi(q,t)\|_\infty\xrightarrow[s\to t]{}0.
\]
\end{proof}

\subsection{Constant-address cells and the parabolic Rauzy limit}

The interpolation becomes especially transparent along constant symbolic rays.

\begin{proposition}[Closed formula for constant-address iterates]\label{prop:constantaddress}
Fix $t\in[0,1]$ and let $a:=2-t$. For every $n\ge1$ and every $(u,v)\in T$,
\[
 f_{1,t}^{\,n}(u,v)=\lambda_{n,t}(u+v)\,(u,v),
\]
where, for $s\in[0,1]$,
\[
\lambda_{n,t}(s)=
\begin{cases}
\displaystyle \frac{1}{a^n+t\,\frac{a^n-1}{1-t}\,s},& 0\le t<1,\\[2ex]
\displaystyle \frac{1}{1+ns},& t=1.
\end{cases}
\]
In particular,
\[
 f_{1,t}^{\,n}(T)=\{u+v\le \sigma_n(t)\},
\]
with
\[
\sigma_n(t)=
\begin{cases}
\displaystyle \frac{1-t}{(2-t)^n-t},& 0\le t<1,\\[2ex]
\displaystyle \frac{1}{n+1},& t=1.
\end{cases}
\]
The analogous formulas hold for $f_{2,t}^{\,n}$ and $f_{3,t}^{\,n}$ by symmetry.
\end{proposition}

\begin{proof}
Let $s=u+v$. Since \eqref{eq:f1chart} preserves every ray from the origin,
\[
 f_{1,t}(u,v)=\frac{1}{a+ts}(u,v).
\]
Hence $f_{1,t}^{\,n}(u,v)=\lambda_n(s)(u,v)$ for some scalar factor $\lambda_n(s)$. Writing $\mu_n(s):=\lambda_n(s)^{-1}$, the recursion becomes
\[
\mu_{n+1}(s)=a\,\mu_n(s)+ts,\qquad \mu_1(s)=a+ts.
\]
Solving this affine recurrence yields
\[
\mu_n(s)=a^n+t(1+a+\cdots+a^{n-1})s=a^n+t\frac{a^n-1}{a-1}s,
\]
which is the stated formula for $t<1$. The case $t=1$ is the limit $a\to1$ and gives $\mu_n(s)=1+ns$.

Finally, $f_{1,t}^{\,n}(T)$ consists exactly of those points on rays from the origin whose radial parameter $s$ is at most the image of $s=1$, namely $\sigma_n(t)$.
\end{proof}

\begin{corollary}[Exponential versus polynomial collapse]\label{cor:collapse}
For each fixed $t<1$ one has $\sigma_n(t)\asymp (2-t)^{-n}$ as $n\to\infty$, whereas
\[
\sigma_n(1)=\frac{1}{n+1}.
\]
Thus the repeated-cell collapse is exponential on every contracting fiber and polynomial at the Rauzy endpoint.
\end{corollary}

This dichotomy is one of the simplest exact signatures of the passage from the affine Sierpi\'nski end to the parabolic Rauzy end; see \cref{fig:scaling}.

\begin{figure}[t]
    \centering
    \includegraphics[width=0.72\textwidth]{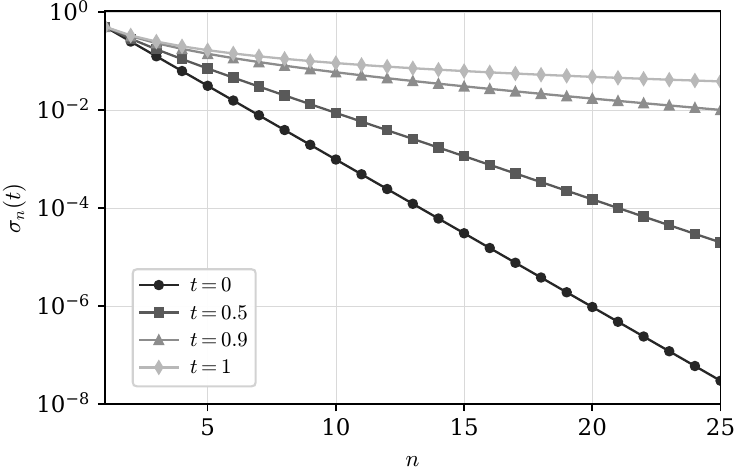}
    \caption{The explicit formula for $\sigma_n(t)$ shows exponential decay for every $t<1$ and polynomial decay at $t=1$.}
    \label{fig:scaling}
\end{figure}

\subsection{The three-dimensional stack}

To assemble the family into a single closed subset of $\R^3$, we place the fiber $K_t$ in the horizontal plane of height
\[
Y=2(1-t),
\]
using the three-dimensional equilateral-triangle embedding
\[
E_t(x_1,x_2,x_3)=\left(\frac{\sqrt 3}{2}(x_3-x_2),\; 2(1-t),\; \frac32 x_1-\frac12\right)\in\R^3.
\]
The bottom face $Y=0$ is the Rauzy fiber $R$, while the top face $Y=2$ is the Sierpi\'nski gasket.

\begin{theorem}[The stack is a product]\label{thm:closedstack}
The set
\[
\mathfrak S:=\{E_t(x):x\in K_t,\ t\in[0,1]\}\subset\R^3
\]
is homeomorphic to $S\times[0,1]$.
\end{theorem}

\begin{proof}
Let $Q=\Sigma/\!\sim$, which is homeomorphic to $S$ by Corollary~\ref{cor:homeos}. By Theorem~\ref{thm:globalcodingcontinuity}, the coding map descends to a continuous map
\[
\bar\Pi:Q\times[0,1]\to T,\qquad \bar\Pi(q,t)=\bar\pi_t(q).
\]
Composing with $E_t$ gives
\[
\Psi:Q\times[0,1]\to \mathfrak S,\qquad \Psi(q,t)=E_t\bigl(\bar\pi_t(q)\bigr).
\]
This map is continuous. It is injective because the second coordinate of $E_t(x)$ is exactly $2(1-t)$, so different parameters $t$ land in different horizontal planes; within a fixed plane, injectivity follows from the fact that the coding relation is the same quotient $Q$ for every $t\in[0,1]$. It is surjective by construction. Since $Q\times[0,1]$ is compact and $\mathfrak S\subset\R^3$ is Hausdorff, $\Psi$ is a homeomorphism. Finally $Q\cong S$.
\end{proof}

The theorem records the topological content of the three-dimensional stack itself: the embedding $E_t$ realizes a genuine compact set homeomorphic to the product of the Sierpi\'nski gasket with an interval, not merely a collection of unrelated planar slices.

\subsection{Hausdorff dimension}

This section identifies the exact Hausdorff dimension of every affine-projective intermediate fiber. The endpoint case $t=1$ is covered by Jurga's theorem for the Rauzy gasket \cite{Jurga2026}; here we show that every $0<t<1$ falls into the same positive-semigroup framework, and that the ambient truncation at dimension $2$ never occurs.

For $t\in(0,1)$ write
\[
 a:=2-t\in(1,2)
\]
and define
\[
A_i(t):=a^{-1/3}M_i(t)\in SL_3(\R),\qquad \Gamma_t:=\langle A_1(t),A_2(t),A_3(t)\rangle.
\]
The projective action of $A_i(t)$ on $\D$ is exactly $f_{i,t}$, since multiplying by a positive scalar does not change projective coordinates. We write $s_A(\Gamma_t)$ for the projective affinity dimension of $\Gamma_t$ in the sense of Jurga \cite{Jurga2026}. Thus, if $\alpha_1(g)\ge \alpha_2(g)\ge \alpha_3(g)$ denote the singular values of $g\in SL_3(\R)$, then for $0\le s\le 2$
\[
\varphi^s(g):=
\begin{cases}
(\alpha_2(g)/\alpha_1(g))^s,& 0\le s\le 1,\\[0.6ex]
(\alpha_2(g)/\alpha_1(g))(\alpha_3(g)/\alpha_1(g))^{s-1},& 1\le s\le 2.
\end{cases}
\]
For a word $w=i_1\cdots i_n$ we write $A_w(t):=A_{i_1}(t)\cdots A_{i_n}(t)$, and we may characterize $s_A(\Gamma_t)$ as the critical exponent of the word series
\[
\sum_{n\ge 1}\ \sum_{|w|=n}\varphi^s(A_w(t)).
\]
This is the formulation used below to verify directly that the ambient cutoff at $2$ never occurs.

\begin{proposition}[Explicit positive conjugation]\label{prop:positiveconj}
Fix $t\in(0,1)$ and write $a=2-t$. For every $c\in(1/a,1)$ sufficiently close to $1/a$, the matrix
\[
P_{a,c}:=\begin{pmatrix}
1&-c&-c\\
-c&1&-c\\
-c&-c&1
\end{pmatrix}
\]
has the property that each conjugate
\[
B_i(t,c):=P_{a,c}^{-1}A_i(t)P_{a,c}
\]
has strictly positive entries.
\end{proposition}

\begin{proof}
Since $A_i(t)=a^{-1/3}M_i(t)$, it is enough to prove positivity for $P_{a,c}^{-1}M_i(t)P_{a,c}$. A direct calculation gives
\[
P_{a,c}^{-1}M_1(t)P_{a,c}
=
\frac{1}{2c^2+c-1}
\begin{pmatrix}
c(a+2)-a & (1-c)(ac-1) & (1-c)(ac-1)\\
c(2c+1-a) & c^2(a+2)-1 & c(ac-1)\\
c(2c+1-a) & c(ac-1) & c^2(a+2)-1
\end{pmatrix},
\]
and the formulas for $M_2(t)$ and $M_3(t)$ are obtained by cyclic permutation of the coordinates. At the boundary value $c=1/a$ this becomes
\[
\begin{pmatrix}
a&0&0\\
1&1&0\\
1&0&1
\end{pmatrix},
\]
while the other two conjugates become its cyclic permutations. Since
\[
2(1/a)^2+1/a-1=\frac{a+2-a^2}{a^2}>0
\]
for $1<a<2$, the denominator stays positive near $1/a$. The entries that are already positive at $c=1/a$ remain positive by continuity, and the zero entries are exactly the ones carrying the factor $ac-1$, so they become positive as soon as $c>1/a$. Hence every conjugate is strictly positive for $c>1/a$ close enough to $1/a$. The same then holds for $B_i(t,c)=a^{-1/3}P_{a,c}^{-1}M_i(t)P_{a,c}$.
\end{proof}

\begin{lemma}[Strong irreducibility]\label{lem:strongirr}
For every $t\in(0,1)$, the semigroup $\Gamma_t$ is strongly irreducible.
\end{lemma}

\begin{proof}
Scalar normalization does not change invariant subspaces, so it is enough to work with
\[
S_t:=\langle M_1(t),M_2(t),M_3(t)\rangle.
\]
Set $a=2-t\in(1,2)$ and $b=(a-1)^{-1}$. Then
\begin{align*}
a^{-n}M_1(t)^n &\longrightarrow
L_1=
\begin{pmatrix}
1&b&b\\
0&0&0\\
0&0&0
\end{pmatrix},\\
a^{-n}M_2(t)^n &\longrightarrow
L_2=
\begin{pmatrix}
0&0&0\\
b&1&b\\
0&0&0
\end{pmatrix},\\
a^{-n}M_3(t)^n &\longrightarrow
L_3=
\begin{pmatrix}
0&0&0\\
0&0&0\\
b&b&1
\end{pmatrix}.
\end{align*}
The row vectors $(1,b,b)$, $(b,1,b)$, and $(b,b,1)$ are linearly independent because their determinant is $(1-b)^2(1+2b)\neq0$. Hence for every nonzero $x\in\R^3$, at least one of $L_1x,L_2x,L_3x$ is nonzero.

Assume that $F$ is a finite $S_t$-invariant set of projective lines. Choose $\ell=[x]\in F$ with $x\neq0$. For some $i$, one has $L_ix\neq0$. Then $M_i(t)^n\ell\in F$ for all $n$, and
\[
[M_i(t)^n x]\longrightarrow [L_i x]=[e_i].
\]
Because $F$ is finite, some infinite subsequence $M_i(t)^{n_k}\ell$ is constant; its limit therefore belongs to $F$. Thus one coordinate line lies in $F$. Applying suitable powers of the other generators shows that all three coordinate lines belong to $F$; for instance,
\[
M_2(t)^n e_1 = e_1+\frac{a^n-1}{a-1}e_2 \longrightarrow [e_2],
\qquad
M_3(t)^n e_1 = e_1+\frac{a^n-1}{a-1}e_3 \longrightarrow [e_3].
\]
Now $[e_1+e_2]=M_2(t)[e_1]\in F$, but
\[
M_3(t)^n(e_1+e_2)=e_1+e_2+2\frac{a^n-1}{a-1}e_3
\]
yields infinitely many distinct projective lines, contradicting finiteness of $F$. So no finite invariant set of lines exists.

The same argument applied to the transpose semigroup $S_t^\top$ excludes finite invariant unions of planes, since such unions correspond by orthogonal complement to finite invariant unions of lines for $S_t^\top$. Therefore $S_t$, and hence $\Gamma_t$, is strongly irreducible.
\end{proof}

\begin{lemma}[Absence of invariant quadratic forms]\label{lem:noquad}
For $t\in(0,1)$ there is no nonzero symmetric matrix $Q$ such that
\[
A_i(t)^\top Q A_i(t)=Q \qquad (i=1,2,3).
\]
\end{lemma}

\begin{proof}
Write $a=2-t\in(1,2)$. Since $A_i(t)=a^{-1/3}M_i(t)$, the displayed equation is equivalent to
\[
M_i(t)^\top Q M_i(t)=a^{2/3}Q \qquad (i=1,2,3).
\]
Let $Q=(q_{jk})_{1\le j,k\le 3}$. Comparing the $(1,1)$-entry in the equation for $i=1$, the $(2,2)$-entry in the equation for $i=2$, and the $(3,3)$-entry in the equation for $i=3$, one obtains
\[
(a^2-a^{2/3})q_{11}=0,\qquad
(a^2-a^{2/3})q_{22}=0,\qquad
(a^2-a^{2/3})q_{33}=0.
\]
Because $1<a<2$, we have $a^2\neq a^{2/3}$, so $q_{11}=q_{22}=q_{33}=0$. The $(2,2)$-entry in the equation for $i=1$ then gives $2q_{12}=0$, the $(3,3)$-entry in the equation for $i=2$ gives $2q_{23}=0$, and the $(1,1)$-entry in the equation for $i=3$ gives $2q_{13}=0$. Hence $Q=0$.
\end{proof}

\begin{proposition}[Zariski density]\label{prop:zariski}
For every $t\in(0,1)$, the semigroup $\Gamma_t$ is Zariski dense in $SL_3(\R)$.
\end{proposition}

\begin{proof}
Let $G$ be the Zariski closure of $\Gamma_t$. We first claim that the identity component $G^\circ$ acts irreducibly on $\R^3$. Indeed, if $W$ were a nonzero proper $G^\circ$-invariant subspace, then the finite union
\[
\bigcup_{g\in G/G^\circ} gW
\]
would be $\Gamma_t$-invariant, contradicting Lemma~\ref{lem:strongirr}.

The unipotent radical of $G^\circ$ is therefore trivial: its fixed-point space is nonzero by Lie--Kolchin, and because the radical is normal that fixed-point space is $G^\circ$-invariant, hence all of $\R^3$; faithfulness then forces the radical to be trivial. So $G^\circ$ is reductive. By Schur's lemma, the connected center of $G^\circ$ acts by scalars, and since $G^\circ\subset SL_3(\R)$, that connected center is trivial. Thus $G^\circ$ is semisimple.

Complexifying the Lie algebra, one obtains a semisimple Lie subalgebra of $\mathfrak{sl}_3(\C)$ with a faithful irreducible $3$-dimensional representation. The only possibilities are $\mathfrak{sl}_3(\C)$ and $\mathfrak{sl}_2(\C)$. If the first occurs, then $G^\circ=SL_3(\R)$ and we are done. In the second case the representation is the symmetric-square representation of $SL_2$, which preserves a nondegenerate quadratic form. That is impossible by Lemma~\ref{lem:noquad}. Therefore $G^\circ=SL_3(\R)$, and hence $G=SL_3(\R)$.
\end{proof}

\begin{proposition}[The affinity dimension is strictly smaller than $2$]\label{prop:sAlesstwo}
For every $t\in(0,1)$ one has
\[
s_A(\Gamma_t)<2.
\]
\end{proposition}

\begin{proof}
Let $a=2-t\in(1,2)$, so $A_i(t)=a^{-1/3}M_i(t)\in SL_3(\R)$. For $s=2$ one has
\[
\varphi^2(A)=\frac{\alpha_2(A)\alpha_3(A)}{\alpha_1(A)^2}=\frac{1}{\alpha_1(A)^3}
\]
for every $A\in SL_3(\R)$, because $\alpha_1(A)\alpha_2(A)\alpha_3(A)=1$.

Each matrix $M_i(t)$ is permutation-conjugate to $M_1(t)$, so it is enough to estimate one of them. Since
\[
M_1(t)^\top \mathbf 1=(a,2,2)^\top,
\]
we obtain
\[
\alpha_1(M_i(t))\ge \frac{\|M_i(t)^\top\mathbf 1\|_2}{\|\mathbf 1\|_2}=\sqrt{\frac{a^2+8}{3}}
\qquad (i=1,2,3).
\]
Therefore
\[
\varphi^2(A_i(t))=\frac{1}{\alpha_1(A_i(t))^3}=\frac{a}{\alpha_1(M_i(t))^3}
\le \frac{3^{3/2}a}{(a^2+8)^{3/2}}.
\]
Summing over the three generators gives
\[
\sum_{i=1}^3 \varphi^2(A_i(t))\le \frac{9\sqrt 3\,a}{(a^2+8)^{3/2}}.
\]
The right-hand side is increasing on $a\in[1,2]$, because its derivative has the sign of $4-a^2$. Hence
\[
\sum_{i=1}^3 \varphi^2(A_i(t))\le \frac{9\sqrt 3\cdot 2}{(2^2+8)^{3/2}}=\frac34<1.
\]
Since the set of generators is finite and $s\mapsto\varphi^s(A_i(t))$ is continuous, there is an $\eta>0$ such that
\[
\sum_{i=1}^3 \varphi^{2-\eta}(A_i(t))<1.
\]
By submultiplicativity of the singular-value potential,
\[
\sum_{|w|=n}\varphi^{2-\eta}(A_w(t))
\le \Bigl(\sum_{i=1}^3\varphi^{2-\eta}(A_i(t))\Bigr)^n.
\]
Thus the defining word series converges at $s=2-\eta$, and consequently $s_A(\Gamma_t)\le 2-\eta<2$.
\end{proof}

\begin{theorem}[Exact dimension of the intermediate fibers]\label{thm:intermediatedimension}
For every $t\in(0,1)$,
\[
\dim_H K_t=s_A(\Gamma_t).
\]
\end{theorem}

\begin{proof}
Write $a=2-t$ and choose $c$ as in Proposition~\ref{prop:positiveconj}. Set
\[
\widetilde\Gamma_t:=P_{a,c}^{-1}\Gamma_t P_{a,c}.
\]
By Proposition~\ref{prop:positiveconj}, the generators of $\widetilde\Gamma_t$ lie in $SL_3(\R)_{>0}$. By Proposition~\ref{prop:firstlevel}, the original system satisfies the strong open set condition with open set $\operatorname{int}(T)$: the sets $f_{i,t}(\operatorname{int}T)$ are pairwise disjoint subsets of $\operatorname{int}T$, and $f_{123,t}(T)\subset \operatorname{int}T$ because the product matrix $M_1(t)M_2(t)M_3(t)$ has strictly positive entries. Projective conjugacy transfers this property to $\widetilde\Gamma_t$. By Proposition~\ref{prop:zariski}, the semigroup $\widetilde\Gamma_t$ is Zariski dense because linear conjugacy changes the Zariski closure only by conjugation.

Jurga's dimension theorem for positive $SL_3$-semigroups satisfying the strong open set condition \cite[Theorem~1.3]{Jurga2026} therefore gives
\[
\dim_H \widetilde K_t = \min\{2,s_A(\widetilde\Gamma_t)\},
\]
where $\widetilde K_t$ denotes the projective attractor of $\widetilde\Gamma_t$. Linear conjugacy changes singular values only by multiplicative constants depending on $P_{a,c}$, so $s_A(\widetilde\Gamma_t)=s_A(\Gamma_t)$. The projective automorphism of $\mathbb{RP}^2$ induced by $P_{a,c}$ is a smooth diffeomorphism, hence bi-Lipschitz on compact subsets in affine charts, and it carries $\widetilde K_t$ onto $K_t$. Consequently,
\[
\dim_H K_t=\min\{2,s_A(\Gamma_t)\}.
\]
The ambient truncation is absent by Proposition~\ref{prop:sAlesstwo}, so in fact $\dim_H K_t=s_A(\Gamma_t)$.
\end{proof}

\begin{corollary}[Dimensions on the full parameter interval]\label{cor:fulldimensions}
Let $\Gamma_1:=\langle M_1(1),M_2(1),M_3(1)\rangle$ be the Rauzy semigroup. Then
\[
\dim_H K_t=
\begin{cases}
\dfrac{\log 3}{\log 2},& t=0,\\[1ex]
s_A(\Gamma_t),& 0<t<1,\\[1ex]
s_A(\Gamma_1),& t=1.
\end{cases}
\]
\end{corollary}

\begin{proof}
The case $t=0$ is the classical Hausdorff-dimension computation for the Sierpi\'nski gasket \cite{Hutchinson1981}. The range $0<t<1$ is exactly Theorem~\ref{thm:intermediatedimension}. The identity at $t=1$ is Jurga's theorem for the Rauzy gasket \cite{Jurga2026}.
\end{proof}

\begin{remark}
The formula for $\dim_H K_t$ is exact, but it remains implicit through the projective affinity dimension $s_A(\Gamma_t)$. Natural next questions are the regularity, monotonicity, and effective computation of the map $t\mapsto s_A(\Gamma_t)$.
\end{remark}

\section{The Sierpi\'nski--Apollonian M\"obius deformation}
This section gives the conformal counterpart to the affine-projective construction. The maps now act on a common closed disk by M\"obius transformations; the affine Sierpi\'nski endpoint occurs at $t=0$, while the endpoint at $t=1$ is parabolic and gives the equilateral Apollonian gasket. Throughout this section the attractors of the M\"obius family are denoted by $\mathcal A_t$, to distinguish them from the affine-projective fibers $K_t$ of the previous section.

The construction theorem below fixes the normalization and records the structural properties needed later. Its proof is intentionally separated into short propositions: endpoint formulas, disk invariance, strict contractivity for $t<1$, and Hausdorff-continuity on compact hyperbolic parameter intervals. The theorem is then proved formally after these ingredients have been established.

\begin{theorem}[The Sierpi\'nski--Apollonian M\"obius family]\label{ap:thm:main}
Let $v_1=\ii$, $v_2=-\frac{\sqrt 3}{2}-\frac{\ii}{2}$, and $v_3=\frac{\sqrt 3}{2}-\frac{\ii}{2}$. For $t\in[0,1]$ define
\[
  R_t=\frac{1-t}{2}+t(2\sqrt 3-3),
  \qquad
  \alpha_t=t\,\frac{1-\sqrt 3}{2},
\]
and for $i\in\{1,2,3\}$ define
\begin{equation}\label{ap:eq:family}
  \Phi_{i,t}(z)
  =R_t\,\frac{z+\alpha_t v_i}{1+\alpha_t\overline{v_i}\,z}+(1-R_t)v_i.
\end{equation}
For $t<1$, let $\mathcal A_t$ denote the attractor supplied by part~\textup{(iii)}, and let $\mathcal A_1$ denote the equilateral Apollonian gasket generated by the endpoint maps in part~\textup{(ii)}. Then the following statements hold.
\begin{enumerate}[label=(\roman*)]
\item At $t=0$ one has
\[
  \Phi_{i,0}(z)=\frac{z+v_i}{2},
\]
so $\mathcal A_0$ is the standard equilateral Sierpi\'nski gasket.
\item At $t=1$ the maps $\Phi_{i,1}$ are the standard parabolic generators of the equilateral Apollonian gasket in the present normalization.
\item For every $t<1$, each $\Phi_{i,t}$ maps the closed unit disk $\Disk$ to itself and is a strict Euclidean contraction on $\Disk$ with Lipschitz constant at most
\[
  L_t=R_t\frac{1-\alpha_t}{1+\alpha_t}<1.
\]
Hence there is a unique compact attractor $\mathcal A_t\subset \Disk$ satisfying
\[
  \mathcal A_t=\Phi_{1,t}(\mathcal A_t)\cup\Phi_{2,t}(\mathcal A_t)\cup\Phi_{3,t}(\mathcal A_t).
\]
\item For every $\varepsilon\in(0,1)$, the family $t\mapsto \mathcal A_t$ is continuous in the Hausdorff metric on $[0,1-\varepsilon]$.
\end{enumerate}
\end{theorem}

Representative finite-level fibers are shown in \cref{ap:fig:slices}; the corresponding first-level carrier disks appear in \cref{ap:fig:carriers}. The theorem concerns the uniformly contracting, hyperbolic regime $t<1$ and identifies the parabolic endpoint at the level of the generators.

\begin{figure}[t]
  \centering
  \includegraphics[width=.98\textwidth]{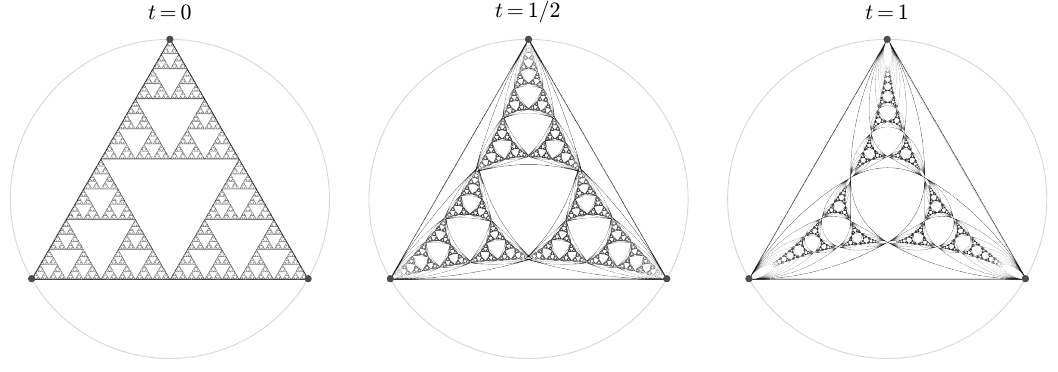}
  \caption{Representative fibers $\mathcal A_t$ at $t=0$, $t=1/2$, and $t=1$.  The black curves are high-depth finite-level images of the three sides of the initial equilateral triangle under the M\"obius semigroup, chosen to stay visually close to the limiting set at journal scale.  The light gray circle is the invariant disk and the light gray triangle records the common fixed vertices.}
  \label{ap:fig:slices}
\end{figure}

\subsection{The one-parameter family}
Let
\[
  \Disk=\{z\in\C:|z|\le 1\}.
\]
The vertices $v_1,v_2,v_3$ lie on the unit circle and form an equilateral triangle. The constants in \eqref{ap:eq:family} are chosen so that the affine endpoint and the parabolic endpoint are both recovered exactly, while the closed unit disk remains a natural carrier set throughout the deformation.

It is convenient to separate the two pieces of \eqref{ap:eq:family}. Set
\[
  M_{i,t}(z)=\frac{z+\alpha_t v_i}{1+\alpha_t\overline{v_i}\,z},
  \qquad
  A_{i,t}(w)=R_t w+(1-R_t)v_i.
\]
Then
\[
  \Phi_{i,t}=A_{i,t}\circ M_{i,t}.
\]
For every $t\in[0,1]$ one has $\alpha_t\in[-(\sqrt 3-1)/2,0]$, hence $|\alpha_t|<1$.

\begin{proposition}[Endpoint formulas]\label{ap:prop:endpoints}
For all $i\in\{1,2,3\}$ and all $t\in[0,1]$ one has $\Phi_{i,t}(v_i)=v_i$. Moreover:
\begin{enumerate}[label=(\alph*)]
\item at $t=0$,
\[
  \Phi_{i,0}(z)=\frac{z+v_i}{2};
\]
\item at $t=1$, each $\Phi_{i,1}$ is parabolic with neutral fixed point $v_i$.
\end{enumerate}
\end{proposition}

\begin{proof}
Since $|v_i|=1$, one has
\[
  M_{i,t}(v_i)=\frac{v_i+\alpha_t v_i}{1+\alpha_t\overline{v_i}v_i}=v_i,
\]
and therefore $\Phi_{i,t}(v_i)=A_{i,t}(v_i)=v_i$.

At $t=0$ one has $R_0=1/2$ and $\alpha_0=0$, so
\[
  \Phi_{i,0}(z)=\frac{1}{2}z+\frac{1}{2}v_i=\frac{z+v_i}{2}.
\]
This is exactly the classical Sierpi\'nski family.

A direct computation gives
\begin{equation}\label{ap:eq:derivative}
  \Phi_{i,t}'(z)=\frac{R_t(1-\alpha_t^2)}{(1+\alpha_t\overline{v_i}z)^2}.
\end{equation}
Evaluating at the fixed point $v_i$ yields
\[
  \Phi_{i,t}'(v_i)=R_t\frac{1-\alpha_t}{1+\alpha_t}.
\]
At $t=1$, the chosen constants satisfy
\[
  R_1\frac{1-\alpha_1}{1+\alpha_1}=1,
\]
so $\Phi_{i,1}$ is parabolic at its fixed point.
\end{proof}

\begin{remark}
After the Euclidean rotation $z\mapsto -\ii z$, the endpoint $t=1$ becomes the familiar Apollonian system generated by
\[
  z\longmapsto \frac{(\sqrt 3-1)z+1}{-z+\sqrt 3+1}
\]
and its two $2\pi/3$-rotates; see \cite{McMullen1998,VytnovaWormell2025}. Thus \eqref{ap:eq:family} is exactly a normalized Sierpi\'nski-to-Apollonian deformation.
\end{remark}

\begin{figure}[t]
  \centering
  \includegraphics[width=.98\textwidth]{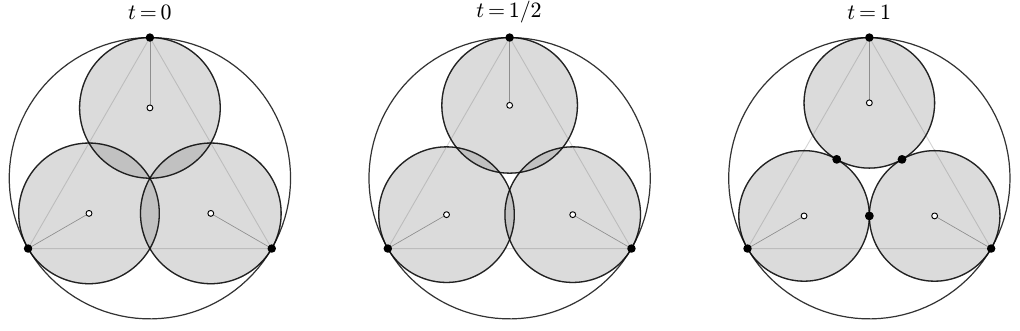}
  \caption{The three carrier disks $\Phi_{1,t}(\Disk)$, $\Phi_{2,t}(\Disk)$, and $\Phi_{3,t}(\Disk)$ inside the unit disk for the same parameter values as in \cref{ap:fig:slices}.  Each carrier is filled in transparent gray, so overlaps appear darker.  The overlap present at the affine Sierpi\'nski endpoint closes up continuously, becoming a mutually tangent Apollonian disk configuration at $t=1$.  Black dots mark the fixed boundary vertices and small white dots mark the disk centers.}
  \label{ap:fig:carriers}
\end{figure}

\subsection{Disk invariance and contractivity}
The formula \eqref{ap:eq:family} is arranged so that the closed unit disk is invariant.

\begin{proposition}[Disk invariance]\label{ap:prop:diskinv}
For every $t\in[0,1]$ and every $i\in\{1,2,3\}$ one has
\[
  \Phi_{i,t}(\Disk)\subset \Disk.
\]
\end{proposition}

\begin{proof}
Because $|\alpha_t|<1$, the map $M_{i,t}$ is a disk automorphism. Hence $M_{i,t}(\Disk)=\Disk$. Now let $w\in\Disk$. Since $|v_i|=1$,
\[
  |A_{i,t}(w)|=|R_t w+(1-R_t)v_i|\le R_t|w|+(1-R_t)|v_i|\le R_t+(1-R_t)=1.
\]
Therefore $A_{i,t}(\Disk)\subset\Disk$, and the same holds for $\Phi_{i,t}=A_{i,t}\circ M_{i,t}$.
\end{proof}

\begin{proposition}[Explicit contraction bound]\label{ap:prop:contraction}
For every $t\in[0,1)$ and every $i\in\{1,2,3\}$, the map $\Phi_{i,t}$ is Lipschitz on $\Disk$ with constant at most
\[
  L_t=R_t\frac{1-\alpha_t}{1+\alpha_t}.
\]
Moreover,
\[
  1-L_t=\frac{(1-t)\bigl(2-(11\sqrt 3-19)t\bigr)}{2\bigl(2-(\sqrt 3-1)t\bigr)}>0,
\]
so $L_t<1$ for all $t<1$ and $L_t\to1$ as $t\to1^{-}$.
\end{proposition}

\begin{proof}
By \eqref{ap:eq:derivative}, for $z\in\Disk$ we have
\[
  |\Phi_{i,t}'(z)|=R_t\frac{1-\alpha_t^2}{|1+\alpha_t\overline{v_i}z|^2}.
\]
Since $\alpha_t\le 0$ and $|z|\le1$, the reverse triangle inequality gives
\[
  |1+\alpha_t\overline{v_i}z|\ge 1+\alpha_t.
\]
Therefore
\[
  |\Phi_{i,t}'(z)|
  \le R_t\frac{1-\alpha_t^2}{(1+\alpha_t)^2}
  =R_t\frac{1-\alpha_t}{1+\alpha_t}=L_t.
\]
The mean value theorem now yields the Lipschitz bound. The displayed identity for $1-L_t$ is a direct algebraic simplification, and every factor in the numerator and denominator is positive for $0\le t<1$.
\end{proof}

\begin{corollary}[Attractors for $t<1$]\label{ap:cor:attractor}
For every $t\in[0,1)$ there is a unique nonempty compact set $\mathcal A_t\subset\Disk$ satisfying
\[
  \mathcal A_t=\Phi_{1,t}(\mathcal A_t)\cup\Phi_{2,t}(\mathcal A_t)\cup\Phi_{3,t}(\mathcal A_t).
\]
At $t=0$, the set $\mathcal A_0$ is the standard equilateral Sierpi\'nski gasket.
\end{corollary}

\begin{proof}
This is the classical contraction mapping theorem for iterated function systems on the complete metric space of nonempty compact subsets of $\Disk$ endowed with the Hausdorff metric \cite{Hutchinson1981}. The endpoint statement at $t=0$ was proved in \Cref{ap:prop:endpoints}.
\end{proof}

\subsection{Continuity away from the parabolic endpoint}
The loss of uniform contraction at $t=1$ is precisely what makes the Apollonian endpoint subtle. Nevertheless the family is well behaved on each truncated interval.

\begin{proposition}\label{ap:prop:continuity}
Fix $\varepsilon\in(0,1)$. Then there exists $0<c_\varepsilon<1$ such that every generator $\Phi_{i,t}$ with $t\in[0,1-\varepsilon]$ is $c_\varepsilon$-Lipschitz on $\Disk$. Consequently the attractor map
\[
  [0,1-\varepsilon]\ni t\longmapsto \mathcal A_t\in\mathcal{K}(\Disk)
\]
into the compact subsets of $\Disk$ endowed with the Hausdorff metric is continuous.
\end{proposition}

\begin{proof}
By \Cref{ap:prop:contraction}, the function $t\mapsto L_t$ is continuous and strictly less than $1$ on the compact interval $[0,1-\varepsilon]$, so it attains a maximum $c_\varepsilon<1$ there. Thus the Hutchinson operators
\[
  \mathcal{H}_t(K)=\Phi_{1,t}(K)\cup\Phi_{2,t}(K)\cup\Phi_{3,t}(K)
\]
are uniformly contracting on $\mathcal{K}(\Disk)$ for $t\in[0,1-\varepsilon]$. Since the maps $\Phi_{i,t}$ depend continuously on $(z,t)$, the operators $\mathcal{H}_t$ depend continuously on $t$ in the Hausdorff metric. Standard fixed-point stability for contractions therefore implies continuity of the unique fixed point $\mathcal A_t$.
\end{proof}

\begin{proof}[Proof of \Cref{ap:thm:main}]
Part~\textup{(i)} follows from \Cref{ap:prop:endpoints,ap:cor:attractor}. Part~\textup{(ii)} follows from the parabolic endpoint calculation in \Cref{ap:prop:endpoints} together with the normalization described immediately after it: after the Euclidean rotation $z\mapsto-\ii z$, the three endpoint maps are the standard equilateral Apollonian generators. Part~\textup{(iii)} is the combination of disk invariance, the explicit contraction bound, and the contraction-mapping argument in \Cref{ap:prop:diskinv,ap:prop:contraction,ap:cor:attractor}. Finally, part~\textup{(iv)} is \Cref{ap:prop:continuity}.
\end{proof}

\subsection{A boundary-normalized Apollonian endpoint}\label{ap:subsec:boundary-normalized}
The equilateral disk model is convenient for the deformation, but there is a second normalization of the endpoint which is better adapted to comparison with the Rauzy side.  In this normalization the distinguished side is the real segment $[0,1]$.

Set
\[
  \lambda:=2+\sqrt3,
  \qquad
  B(z):=\frac{1+\lambda\ii}{2}\,\frac{z-v_2}{z+\lambda\ii},
\]
and define
\[
  \mathcal A^{\flat}:=B(\mathcal A_1).
\]
We call $\mathcal A^{\flat}$ the boundary-normalized Apollonian endpoint.

\begin{theorem}[Boundary-normalized Apollonian model]\label{ap:thm:boundary-model}
Let $p=(1+\ii)/2$. The following statements hold.
\begin{enumerate}[label=(\roman*)]
\item The M\"obius map $B$ sends
\[
  v_2\longmapsto 0,
  \qquad
  v_3\longmapsto 1,
  \qquad
  v_1\longmapsto p.
\]
Moreover it sends the side circle
\[
  \mathcal C_{23}:=\{z\in\C: |z+2\ii|=\sqrt3\}
\]
to the real line, and the side arc from $v_2$ to $v_3$ not containing the pole of $B$ to the segment $[0,1]$.
\item Up to the relabelling $0\leftrightarrow 2$, $1\leftrightarrow 3$, and $\ast\leftrightarrow 1$, the endpoint generators are conjugated by $B$ to
\begin{equation}\label{ap:eq:boundary-generators}
  G_0(z)=\frac{z}{z+1},
  \qquad
  G_1(z)=\frac{1}{2-z},
  \qquad
  G_\ast(z)=\frac{z-\ii}{2z-1-2\ii}.
\end{equation}
That is,
\[
  G_0=B\circ\Phi_{2,1}\circ B^{-1},
  \qquad
  G_1=B\circ\Phi_{3,1}\circ B^{-1},
  \qquad
  G_\ast=B\circ\Phi_{1,1}\circ B^{-1}.
\]
Consequently
\[
  \mathcal A^{\flat}=G_0(\mathcal A^{\flat})\cup G_1(\mathcal A^{\flat})\cup G_\ast(\mathcal A^{\flat}).
\]
\item The three side circles of the equilateral Apollonian endpoint are mapped to
\[
  \R\cup\{\infty\},
  \qquad
  \partial\overline{D}\!\left(\frac{\ii}{2},\frac12\right),
  \qquad
  \partial\overline{D}\!\left(1+\frac{\ii}{2},\frac12\right).
\]
Thus $\mathcal A^{\flat}$ is exactly the Apollonian gasket in the upper-half-plane normalization whose base side is $[0,1]$ and whose two adjacent outer circles have radius $1/2$ and centers $\ii/2$ and $1+\ii/2$.
\item On the distinguished side one has
\[
  G_0([0,1])=[0,1/2],
  \qquad
  G_1([0,1])=[1/2,1],
\]
and therefore
\[
  [0,1]=G_0([0,1])\cup G_1([0,1]).
\]
\end{enumerate}
\end{theorem}

\begin{proof}
The first two images of vertices are immediate from the definition of $B$.  For the third, substituting $v_1=\ii$ gives
\[
  B(v_1)=\frac{1+\lambda\ii}{2}\,
        \frac{\ii-v_2}{\ii+\lambda\ii}
       =\frac{1+\ii}{2}=p.
\]
The pole of $B$ is $-\lambda\ii=-(2+\sqrt3)\ii$.  Since
\[
  \bigl|-(2+\sqrt3)\ii+2\ii\bigr|=\sqrt3,
\]
this pole lies on $\mathcal C_{23}$.  Hence $B$ maps $\mathcal C_{23}$ to a generalized circle through $0$, $1$, and $\infty$, namely the real line.  The arc from $v_2$ to $v_3$ not containing the pole is mapped to $[0,1]$; for instance the other point of $\mathcal C_{23}$ on the imaginary axis, $-(2-\sqrt3)\ii$, is mapped to $1/2$.

A direct substitution using the formula for $B^{-1}$ gives
\[
  B\circ\Phi_{2,1}\circ B^{-1}(z)=\frac{z}{z+1},
  \qquad
  B\circ\Phi_{3,1}\circ B^{-1}(z)=\frac{1}{2-z},
\]
and
\[
  B\circ\Phi_{1,1}\circ B^{-1}(z)=\frac{z-\ii}{2z-1-2\ii}.
\]
This proves the conjugacy formulas and the invariance equation for $\mathcal A^{\flat}$.

It remains only to identify the two adjacent side circles.  In the equilateral endpoint the three side circles are
\[
  \mathcal C_{ij}:=\{z\in\C: |z+2v_k|=\sqrt3\},
  \qquad \{i,j,k\}=\{1,2,3\}.
\]
They are pairwise tangent, since their centers are $-2v_1,-2v_2,-2v_3$ and their common radius is $\sqrt3$.  The conjugacy formulas above show that $\mathcal C_{23}=B^{-1}(\R\cup\{\infty\})$ is invariant under the two endpoint branches based at $v_2$ and $v_3$; by cyclic symmetry the same statement holds for the other two side circles.  The circle $\mathcal C_{23}$ has already been sent to the real line.  The two remaining side circles are therefore sent to two circles tangent to the real line at $0$ and $1$ and tangent to each other at $p=(1+\ii)/2$.  The unique such pair is
\[
  \partial\overline{D}\!\left(\frac{\ii}{2},\frac12\right),
  \qquad
  \partial\overline{D}\!\left(1+\frac{\ii}{2},\frac12\right).
\]
Finally, the interval formulas follow directly from \eqref{ap:eq:boundary-generators}: for $x\in[0,1]$,
\[
  G_0(x)=\frac{x}{x+1}\in[0,1/2],
  \qquad
  G_1(x)=\frac{1}{2-x}\in[1/2,1],
\]
and the two ranges meet only at $1/2$.
\end{proof}

\begin{corollary}[Canonical Rauzy--Apollonian identification fixing the base side]\label{ap:cor:rauzy-apollonian-homeo}
Let $R=K_1$ be the Rauzy gasket in the simplex chart of \cref{rem:rauzyendpoint}, and let
\[
  J:=\{(1-x,x):0\le x\le1\}\subset T
\]
be the side joining $e_2$ to $e_3$.  There is a canonical homeomorphism
\[
  \Theta:R\longrightarrow \mathcal A^{\flat}
\]
obtained by matching the common Sierpi\'nski address quotient.  Moreover, under the coordinate $x$ on $J$,
\[
  \Theta(1-x,x)=x\qquad(0\le x\le1).
\]
\end{corollary}

\begin{proof}
The Rauzy endpoint is coded by the standard Sierpi\'nski quotient $Q=\Sigma/\!\sim$ by \cref{rem:rauzyendpoint}.  The boundary-normalized Apollonian endpoint has the same quotient: the three first-level curvilinear cells have pairwise intersections at the three vertices, and by applying the M\"obius maps $G_0,G_1,G_\ast$ the same incidence pattern holds in every cylinder.  The diameters of the cylinder cells tend to zero.  If an infinite word is eventually constant, this follows from the parabolic normal forms; for instance
\[
  G_0^n(z)=\frac{z}{nz+1},
  \qquad
  G_1^n(z)=1-\frac{1-z}{n(1-z)+1},
\]
and the third branch is conjugate to the same form at the fixed point $p$.  If the word changes symbol infinitely often, then along arbitrarily long prefixes the corresponding cell is mapped into a compact subcell away from the neutral fixed point, where the spherical derivatives are uniformly smaller than one.  Hence the only multiple addresses are the usual vertex identifications
\[
  wij^\infty\sim wji^\infty\qquad(i\ne j),
\]
so $\mathcal A^{\flat}$ is also canonically homeomorphic to $Q$.  Composing the two quotient parametrizations gives $\Theta$.

It remains to check the statement on the distinguished side.  Parametrize $J$ by $\theta(x)=(1-x,x)$.  At the Rauzy endpoint $t=1$, the restrictions of the two side branches are
\[
  f_{2,1}(\theta(x))=\theta\!\left(\frac{x}{x+1}\right),
  \qquad
  f_{3,1}(\theta(x))=\theta\!\left(\frac{1}{2-x}\right).
\]
These are exactly the restrictions of $G_0$ and $G_1$ to $[0,1]$.  Therefore the address map on the Rauzy side and the address map on the Apollonian side are the same one-dimensional coding map.  The quotient-matching homeomorphism is consequently the identity on the side $[0,1]$.
\end{proof}

\begin{figure}[t]
  \centering
  \includegraphics[width=.90\textwidth]{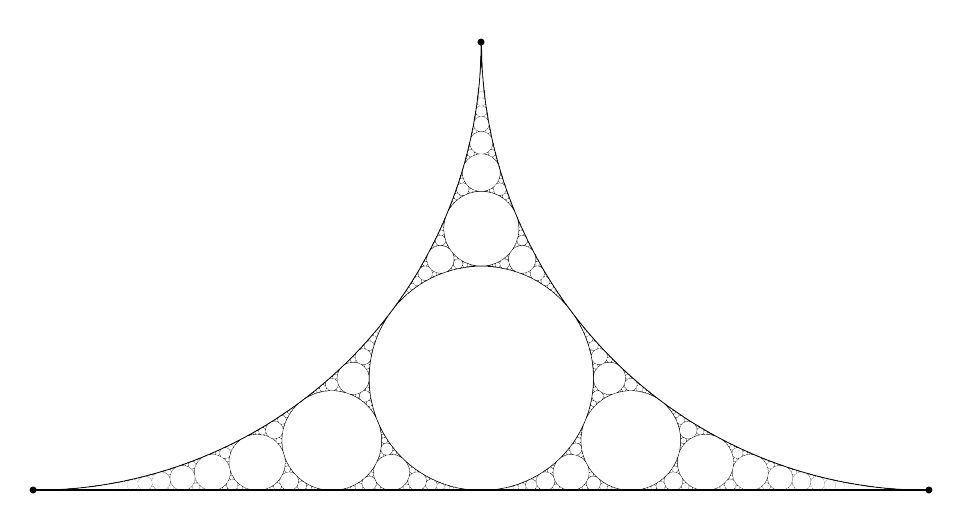}
  \caption{High-depth finite-level approximation of the boundary-normalized Apollonian endpoint $\mathcal A^{\flat}$.  The distinguished side is the segment $[0,1]$; the two adjacent outer circles have centers $\ii/2$ and $1+\ii/2$ and radius $1/2$.  The figure is generated directly from exact M\"obius images of the boundary arcs in \eqref{ap:eq:boundary-generators}.}
  \label{ap:fig:boundary-normalized}
\end{figure}

\begin{remark}[A side-fixed visual interpolation]\label{ap:rem:side-fixed-interpolation}
The boundary-normalized endpoint also gives a compact formula for the supplementary side-fixed visualization.  Let $p=(1+\ii)/2$. Put
\[
  T^{\flat}:=\operatorname{conv}\{0,1,p\}.
\]
For $0\le s\le1$ define
\[
  H_{0,s}(z)=\frac{z}{2-s+sz},
  \qquad
  H_{1,s}(z)=1-\frac{1-z}{2-s+s(1-z)},
\]
and
\[
  H_{\ast,s}(z)
  =p+\frac{z-p}{2-s+2\ii s(z-p)}.
\]
At $s=0$ these are the affine Sierpi\'nski branches on $T^{\flat}$,
\[
  z\mapsto \frac z2,
  \qquad
  z\mapsto \frac{1+z}{2},
  \qquad
  z\mapsto \frac{z+p}{2},
\]
while at $s=1$ they are exactly the boundary-normalized Apollonian generators $G_0,G_1,G_\ast$ from \eqref{ap:eq:boundary-generators}.  The restrictions of $H_{0,s}$ and $H_{1,s}$ to the base segment interpolate through the same fractional transformations that appear on the distinguished Rauzy side.  This side-fixed interpolation is used in the interactive companion to display the boundary-normalized comparison; the formal results above do not require any additional dynamical assertion about this auxiliary visualization.
\end{remark}

\begin{remark}
Finite-level computations strongly suggest that $\mathcal A_t\to \mathcal A_1$ in the Hausdorff metric as $t\to1^{-}$, where $\mathcal A_1$ is the classical equilateral Apollonian gasket generated by the parabolic endpoint maps $\Phi_{i,1}$. This endpoint continuity is not used in the results proved below.
\end{remark}

\subsection{Dimension theory of the hyperbolic fibers}
For $t\in[0,1)$ let
\[
  \Sigma=\{1,2,3\}^{\mathbb N},
  \qquad
  \sigma((x_n)_{n\ge1})=(x_{n+1})_{n\ge1},
\]
and let $\pi_t:\Sigma\to \mathcal A_t$ denote the coding map
\[
  \pi_t(x)=\lim_{n\to\infty} \Phi_{x_1,t}\circ\cdots\circ\Phi_{x_n,t}(0).
\]
Let $\mathcal P=\{[1],[2],[3]\}$ be the one-step cylinder partition of $\Sigma$, and let $\gamma_{\C}$ denote the Borel $\sigma$-algebra of $\C$. For $m\in M_\sigma(\Sigma)$ define the projection entropy
\[
  h_{\pi_t}(\sigma,m)
  := H_m(\mathcal P\mid \sigma^{-1}\pi_t^{-1}\gamma_{\C})
   - H_m(\mathcal P\mid \pi_t^{-1}\gamma_{\C}),
\]
and the Lyapunov exponent
\[
  \chi_t(m)
  := -\int_\Sigma \log\bigl|\Phi'_{x_1,t}(\pi_t(\sigma x))\bigr|\,dm(x).
\]

\begin{theorem}[Exact Hausdorff dimension of every hyperbolic fiber]\label{ap:thm:exact-dim}
For every $t\in[0,1)$ the family $\{\Phi_{1,t},\Phi_{2,t},\Phi_{3,t}\}$ is a conformal IFS on the closed unit disk. Consequently:
\begin{enumerate}[label=(\roman*)]
\item for every ergodic $m\in M_\sigma(\Sigma)$, the projected measure $\mu_{t,m}:=m\circ\pi_t^{-1}$ is exact-dimensional and
\[
  \dimH \mu_{t,m}=\frac{h_{\pi_t}(\sigma,m)}{\chi_t(m)};
\]
\item the attractor itself satisfies the exact variational formula
\[
  \dimH \mathcal A_t = \dim_B \mathcal A_t
  = \sup_{m\in M_\sigma(\Sigma)} \frac{h_{\pi_t}(\sigma,m)}{\chi_t(m)}
  = \sup_{\substack{m\in M_\sigma(\Sigma)\\ m\text{ ergodic}}} \dimH\mu_{t,m}.
\]
\end{enumerate}
\end{theorem}

\begin{proof}
Fix $t<1$. By \Cref{ap:prop:contraction}, each $\Phi_{i,t}$ is a strict contraction of $\Disk$. If $\alpha_t=0$, then $\Phi_{i,t}$ is affine and entire. If $\alpha_t\ne0$, then $\Phi_{i,t}$ is a M\"obius map with pole at $-1/(\alpha_t\overline{v_i})$, which lies strictly outside $\Disk$ because $|\alpha_t|<1$. Hence, in every case, each branch extends holomorphically to an open neighborhood of $\Disk$. Its derivative is given by \eqref{ap:eq:derivative}, so it never vanishes there. Therefore the family is a conformal IFS in the sense of Feng and Hu \cite[Definition~2.9]{FengHu2009}. Part~(i) is exactly \cite[Theorem~2.8]{FengHu2009}, and part~(ii) is \cite[Theorem~2.13]{FengHu2009}.
\end{proof}

\begin{corollary}[The Sierpi\'nski endpoint]\label{ap:cor:sierp-dim}
At $t=0$ one has
\[
  \dimH(\mathcal A_0)=\frac{\log 3}{\log 2}.
\]
\end{corollary}

\begin{proof}
At $t=0$ the branches are the three similarities $z\mapsto(z+v_i)/2$ with strong separation on the open triangle interior. Hence the classical similarity-dimension formula applies.
\end{proof}

\begin{corollary}[The Apollonian endpoint]\label{ap:cor:apol-dim}
At $t=1$ one has
\[
  \dimH(\mathcal A_1)=1.3056867280498771846459862068510408911060\ldots \pm 10^{-129}.
\]
\end{corollary}

\begin{proof}
At $t=1$ the family is the classical equilateral Apollonian parabolic system, so this is exactly the theorem of Vytnova and Wormell \cite{VytnovaWormell2025}. Earlier transfer-operator approximations go back to McMullen \cite{McMullen1998}.
\end{proof}

\begin{remark}
\Cref{ap:thm:exact-dim} gives an exact formula for every intermediate fiber, but it is variational rather than closed-form. What remains outside the scope of this section is a more explicit evaluation of the supremum as a function of $t$, analogous to the much subtler parabolic computation at $t=1$.
\end{remark}

\section{Comparison of the projective and conformal endpoints}
The two parameterized families studied above form a three-stage comparison scheme
\[
\text{Rauzy gasket}\quad\longleftrightarrow\quad
\text{Sierpi\'nski gasket}\quad\longleftrightarrow\quad
\text{equilateral Apollonian gasket}.
\]
The first arrow is affine-projective: the maps act on the simplex and become the Rauzy inverse branches at a projective endpoint. The second arrow is conformal: the maps are M\"obius transformations preserving a common carrier disk for $t<1$ and converge to parabolic Apollonian generators at the endpoint. The finite-level approximants have different meanings in the two families. In the affine-projective family, finite-level cylinder complements make the non-homogeneous projective scaling visible; in the conformal family, the same triangular combinatorics is realized by gaps whose boundaries deform toward mutually tangent circles.

Thus the two deformations are best compared through the exact maps and finite-level approximants, rather than identified dynamically. The affine-projective family records a degeneration from the Sierpi\'nski gasket to the Rauzy gasket; the conformal family records a deformation from the Sierpi\'nski gasket to the equilateral Apollonian gasket.

The boundary-normalized endpoint in \cref{ap:thm:boundary-model,ap:cor:rauzy-apollonian-homeo} sharpens this comparison. The Rauzy and Apollonian endpoints are not conformally or projectively the same dynamical system, but their distinguished sides can be represented by the same interval and by the same two fractional branches
\[
  x\longmapsto \frac{x}{x+1},
  \qquad
  x\longmapsto \frac{1}{2-x}.
\]
The resulting homeomorphism fixes the common side pointwise and separates the exact topological equivalence from the different projective and conformal geometries carried by the two endpoints.  In this setting, projective, affine, and conformal triangular fractals can be compared within a single exact framework while retaining the distinct nature of their endpoint dynamics.

\section*{Supplementary material}
A supplementary interactive implementation accompanies the article:
\begin{center}
\href{https://complextrees.com/rauzy-sierpinski-apollonian}{\nolinkurl{complextrees.com/rauzy-sierpinski-apollonian}}.
\end{center}
It displays finite-level approximants of the affine-projective, conformal, unified, and boundary-normalized side-fixed views described above, and provides STL export for selected three-dimensional embeddings. The mathematical definitions and results are those stated in the article.

\section*{Funding}
The author was supported by the Ministerio de Ciencia, Innovaci\'on y Universidades, Agencia Estatal de Investigaci\'on, and FEDER, UE, under Grant PID2023-146424NB-I00, and by the Universitat de Girona and Banco Santander Grant Programme for Researchers in Training (IFUdG 2022--2024). The ICERM residency during the \emph{Illustrating Mathematics} program was supported by National Science Foundation Grant DMS1439786 and Alfred P. Sloan Foundation award G-2019-11406.

\section*{Acknowledgements}
The author is grateful to Pierre Arnoux for suggesting at ICERM 2019 the problem of constructing a three-dimensional model relating the Rauzy gasket to the Sierpi\'nski and Apollonian gasket geometries, and for reading an early version of the affine-projective construction. Arnoux's comments led in particular to the clarification of finite-level approximants in \cref{fig:crosssections} and to the boundary-normalized Apollonian model of \cref{ap:subsec:boundary-normalized}. The author thanks David Juher and Joan Salda\~na for helpful discussions. He thanks the organizers of \emph{Integrating Research and Illustration in Number Theory} (Institut Henri Poincar\'e, Paris, March 23--27, 2026)---Xavier Caruso, Ellen Eischen, Catherine Hsu, and Katherine E. Stange---for the setting in which he met Pierre Arnoux again and returned to this project. He also thanks the organizers of \emph{Complex dynamics: connections to other fields} (Ch\k{e}ciny, 27--31 March 2023), part of the mini-semester \emph{Modern holomorphic dynamics and related fields}, where Polina Vytnova (University of Surrey) gave the presentation \href{https://www.mimuw.edu.pl/~holdyn23/pages/presentations/presentation-Vytnova.pdf}{\emph{32 digits of the Hausdorff dimension of the Apollonian gasket}}; this presentation was an early motivation for the Apollonian side of the present comparison.

\bibliographystyle{amsplain}
\bibliography{references}

\end{document}